\documentclass[12pt]{article}
\usepackage{graphicx}
\usepackage{amsmath}
\usepackage{amsfonts}
\usepackage{amssymb}
\usepackage{enumerate}
\newtheorem{theorem}{Theorem}

\newtheorem{corollary}[theorem]{Corollary}

\newtheorem{definition}[theorem]{Definition}

\newtheorem{lemma}[theorem]{Lemma}

\newenvironment{proof}[1][Proof]{\textbf{#1.} }{\ \rule{0.5em}{0.5em}}

\begin{document}

\title{Energy dissipation and self-similar solutions for an unforced inviscid dyadic model}
\author{D. Barbato, F. Flandoli, F. Morandin}
\maketitle
\begin{abstract}
A shell-type model of an inviscid fluid, previously considered in the
literature, is investigated in absence of external force. Energy dissipation
of positive solutions is proved and decay of energy like $t^{-2}$ is
established. Self-similar decaying positive solutions are introduced and
proved to exist and classified. Coalescence and blow-up are obtained as a
consequence, in the class of arbitrary sign solutions.
\end{abstract}

\section{Introduction}

The following system of differential equations{\
\begin{equation}
\left\{
\begin{array}
[c]{ll}%
X_{0}(t)=0 & \qquad\forall t\geq0\\
\dot{X_{n}}(t)=k_{n-1}X_{n-1}^{2}(t)-k_{n}X_{n}(t)X_{n+1}(t) & \qquad\forall
t\geq0,\quad\forall n\geq1
\end{array}
\right.  \label{model}%
\end{equation}
where }$k_{n}=2^{n}$,{\ }has been introduced as a simplified model of 3D Euler
evolution in order to investigate a number of properties which are out of
reach at present for more realistic models of fluid dynamics. Let us mention
in particular the works \cite{FriPav} \cite{KatPav}, \cite{KiZlat},
\cite{CFP1}, \cite{Ces}, \cite{CFP2} and references therein, devoted to this
model and variants of it. This model differs from other dyadic or shell
models, the analysis of which is more difficult and less explicit, see among
many others \cite{biferale}, \cite{Ga}, \cite{BLPTiti},
\cite{constantin-titi3}, \cite{BBBF};\ system~\eqref{model} has a basic
monotonicity property that makes it more tractable.

Among the many interesting properties proved in the above mentioned
works, let us recall: i) the dissipation of energy, in spite of the
fact that formally the equation is conservative; ii) the blow-up of
solutions in certain topologies, although the same solutions are
global in larger spaces; iii) the relation with Kolmogorov (K41)
scaling law. Our aim is to add some contribution to the understanding
of these problems. The basic difference between the previous works
dealing with energy dissipation and the present one is that a non-zero
force was imposed until now, while we investigate the case without
force, which contains a number of new phenomena.

About energy dissipation, the results known until now have the following
form. A constant positive force $f>0$ is added to the first mode
\[
\dot{X_{1}}(t)=-k_{1}X_{1}(t)X_{2}(t)+f
\]
and it is proved that there exists a unique fixed point, having Kolmogorov
scaling, which attracts \textit{exponentially} all positive solutions
(solutions with non negative components), in the topology $l^{2}$. This
implies that energy is dissipated. Positive solutions are better motivated in
comparison with fluid dynamic equations (see \cite{CFP1}).

Although the case $f>0$ and the fixed point are very interesting, it
is also of interest to analyze the free (\textit{unforced}) dynamic,
namely system~\eqref{model} without any forcing term. Physically, if
we accept that a dyadic model like~\eqref{model} may describe
something of turbulence, the unforced case would correspond to free
decaying turbulence, a widely observed phenomenon, see \cite{Fri} and
references therein. The results of energy dissipation from the
previous papers do not extend to this case, they really require $f>0$
and make use of the non-trivial fixed point in the computations. For
simplicity, one would conjecture the case $f=0$ to be similar
(exponential decay to zero), maybe with a different proof, but it
turns out this is not the case.

For the unforced case, namely system~\eqref{model}, we first prove the
following result: all finite energy positive solutions have energy
which decays to zero as $t\rightarrow\infty$. Nevertheless, for
sufficiently regular initial conditions, energy is conserved for small
times, as was shown for instance in \cite{FriPav}, \cite{KatPav} and
\cite{Wal}.

Even a large class of solutions starting with infinite energy
immediately enter $l^{2}$, namely they immediately get finite energy,
and then continue their process of dissipation to zero.

However, the decay of energy to zero is no more exponential. By means of a
scaling argument we prove an upper bound on the decay of the energy of the
order $t^{-2}$. Then we prove a weak form of lower bound of the same order,
which includes in particular the property $\int_{0}^{\infty}X_{n}\left(
t\right)  dt=\infty$ for all $n$ greater than some $n_{0}$. Thus exponential
decay is ruled out.

Then we investigate positive \textit{self-similar} solutions, of the form
\begin{equation}
X_{n}\left(  t\right)  =\frac{a_{n}}{t-t_{0}},\quad t\geq t_{0}%
.\label{forma delle self similar}%
\end{equation}
\noindent\quad\noindent\noindent We prove the existence of such solutions in
the space $l^{2}$. The proof is highly non-trivial. The energy of these
solutions decays exactly as $t^{-2}$ and we conjecture that this should be the
case for all positive solutions.

The existence of such self-similar solutions relates also to the problem of
blow-up and coalescence, but in a rather controversial way from the viewpoint
of the physical interest. If we decide that only solutions with positive
components have physical meaning, our self-similar solutions do not contribute
to the problem of blow-up. If on the contrary we consider system~\eqref{model}
as a nonlinear model to be understood for any kind of initial condition, we
have proved (by a simple inversion of time) that there exist solutions defined
on the time interval $\left(  -\infty,t_{0}\right)  $ of the form
\[
X_{n}\left(  t\right)  =\frac{a_{n}}{t-t_{0}},\quad t\leq t_{0}%
\]
with $\left(  a_{n}\right)  \in l^{2}$. The components $X_{n}\left(  t\right)
$ are all negative. These solutions blow up in finite time,
\textit{irrespective of the topology}, in the sense that all components blow
up in finite time. This result is much stronger than the blow-up results of
the previous literature on dyadic models. 

The existence of positive self-similar solutions also implies coalescence: we
prove that at every point of a self-similar solution there is the coalescence
of at least another solution (which cannot be positive). 

Let us finally mention other properties of the self-similar solutions we
construct and some open problems.

Given $t_{0}$, we prove that there is a \textit{unique} solution of the 
form~\eqref{forma delle self similar} in $l^{2}$, with
\textit{strictly positive} components. We also prove that the
components $a_{n}$ decay as
\[
a_{n}\sim C_02^{-n/3}%
\]
so, in a sense, these are `Kolmogorov type decaying solutions' \cite{K41}.
There are two degrees of freedom, however, in these solutions. One is the
given value of $t_{0}$. The other is that we could take
\[
a_{1}=...=a_{n_{0}}=0
\]
and prove that there exists a unique (for given $t_{0}$) self-similar solution
with these first components equal to zero and all the others strictly
positive. In this case, closer inspection gives us
\[
a_{n}\sim C_{n_{0}}2^{-n/3}%
\]
where
\[
C_{n_{0}}=2^{-2n_{0}/3}R^{-1},\qquad\qquad R\in(4/5,1)
\]
(numerical experiments give us $R\approx0.885765931$).

Thus the picture is that in $l^{2}$ travel a family of self-similar solutions
depending on the continuous parameter $t_{0}$ and the discrete parameter
$n_{0}$.

What happens to all other solutions? Do they approach this set of special
solutions? In this case, does a generic solution select one particular
self-similar solution and get closer and closer to it, or does it slowly shift
from one self-similar solution to the other? A sufficiently fast convergence
to one self-similar solution would imply that Kolmogorov K41 scaling holds
true for this simple dyadic model. But a slow convergence or a shift along
different self-similar solutions could produce small deviations from
Kolmogorov scaling. Further research is necessary to clarify these issues.

\section{Concepts of solution, their existence, positivity and energy inequality}

To start, let us state rather general existence theorems. Their proofs are
classical, compared to the previous literature, but at least the case of
infinite initial energy is new, so we give the details in the Appendix. We
denote by $H$ the space $l^{2}$, namely the space of all square integrable
sequences of real numbers. It is a Hilbert space with the obvious inner
product $\left\langle u,v\right\rangle _{H}=\sum_{n=1}^{\infty}u_{n}v_{n}$,
where $u=\left(  u_{n}\right)  _{n\in\mathbb{N}}$, $v=\left(  v_{n}\right)
_{n\in\mathbb{N}}$, $u,v\in H$. The corresponding norm in $H$ will be denoted
by $\left|  \cdot\right|  _{H}$. We denote the space of all sequences of real
numbers by $\mathbb{R}^{\mathbb{N}}$, and its subset of all non-negative
sequences by $\mathbb{R}_{+}^{\mathbb{N}}$.

In the sequel we shall often use the term ``energy'' for the quantity $\left|
X\right|  _{H}^{2}$, for an element $X\in H$ or also $X\in\mathbb{R}%
^{\mathbb{N}}$ (it may have infinite energy).

\begin{definition}
\label{def1}Given $X^{0}\in\mathbb{R}^{\mathbb{N}}$, we call componentwise
solution of system~\eqref{model} with initial condition $X^{0}$ any sequence
$X=\left(  X_{n}\left(  \cdot\right)  \right)  _{n\in\mathbb{N}}$ of
continuously differentiable functions $X_{n}\left(  \cdot\right)
:[0,\infty)\rightarrow\mathbb{R}$ such that $X_{n}\left(  0\right)  =X_{n}%
^{0}$ for all $n\in\mathbb{N}$ and all equations in system~\eqref{model} are
satisfied. If $X\left(  t\right)  \in$ $\mathbb{R}_{+}^{\mathbb{N}}$ for all
$t\geq0$, we call it a positive componentwise solution.
\end{definition}

If $X$ is a componentwise solution, from system~\eqref{model} we have
\begin{equation}
X_{n}(t)=X_{n}^{0}e^{-k_{n}\int_{0}^{t}X_{n+1}(r)dr}+\int_{0}^{t}e^{-k_{n}%
\int_{s}^{t}X_{n+1}(r)dr}k_{n-1}X_{n-1}^{2}\left(  s\right)  ds.
\label{variation of constants}%
\end{equation}
This identity will be used quite often.

\begin{definition}
\label{def2}We say that a componentwise solution $X$ has finite energy for
positive times if $X\left(  t\right)  \in H$ for all $t>0$. If also $X^{0}\in
H$, we call $X$ a finite energy solution.
\end{definition}

\begin{theorem}
\label{teo existence 1}Given $X^{0}\in\mathbb{R}_{+}^{\mathbb{N}}$, any
componentwise solution of system~\eqref{model} with initial condition $X^{0}$
is positive. At least one such solution exists. Moreover, any such solution
has the following properties:
\begin{enumerate}[i)]
\item for every $n\geq1$ and $t\geq0$ we have
\[
\frac d{dt}\sum_{j=1}^{n} X^2_{j}\left(  t\right)
=-k_nX_n^2(t)X_{n+1}(t)
\]
and hence
\begin{equation}\label{phi monotona}
\sum_{j=1}^{n} X^2_{j}\left(  t\right) \leq\sum_{j=1}%
^{n}\left(  X_{j}^{0}\right)  ^{2}
\end{equation}

\item if $X_{n}^{0}>0$ for some $n\geq1$, then $X_{m}(t)>0$ for all $m\geq n$
and all $t>0$.
\end{enumerate}
\end{theorem}

\begin{theorem}
\label{teo existence 2}For every $X^{0}\in H$, there exists at least one
finite energy solution of system~\eqref{model} with initial condition $X^{0}$,
with the property
\begin{equation}
\left|  X\left(  t\right)  \right|  _{H}\leq\left|  X(s)\right|  _{H}%
\quad\text{for all }0\leq s\leq t.\label{energy inequality}%
\end{equation}
Moreover, if $X^{0}\in H\cap\mathbb{R}_{+}^{\mathbb{N}}$, then all
componentwise solutions are finite energy and satisfy~\eqref{energy inequality}.
\end{theorem}

See the proofs in the Appendix.

\section{Energy dissipation}

System~\eqref{model} is \emph{formally} conservative:
\[
\frac{1}{2}\frac{d}{dt}\left|  X\right|  _{H}^{2}=\sum_{n=1}^{\infty}\left(
k_{n-1}X_{n}X_{n-1}^{2}-k_{n}X_{n}^{2}X_{n+1}\right)  =0
\]
by a simple rearrangement of the series and the condition $X_{0}=0$.
This rearrangement is rigorous if the solutions live in a sufficiently
regular space (see for example \cite{CFP1}). Such kind of regularity may hold for small times if the
initial condition is very regular itself (see \cite{FriPav}), but for
sufficiently large times we prove that solutions dissipate energy. The
intuitive mechanism is a very fast shift of energy from small to large
$n$ components.

We give two results of energy dissipation for positive solutions:
infinite initial energy becomes finite immediately;\ the energy of a
finite energy solution tends to zero as $t\rightarrow\infty$. Although
the degree of infinity of the energy of initial conditions can be
generalized, for the simplicity of statements we restrict ourselves to
$X^{0}$ of class $l^{\infty}$: the norm $||X^{0}%
||_{\infty}:=\sup_{n}\left| X_{n}^{0}\right| $ is finite.

\begin{theorem}
\label{teo dissip 1}Assume $X^{0}\in l^{\infty}\cap\mathbb{R}_{+}^{\mathbb{N}%
}$ and let $X$ be a positive componentwise solution of system~\eqref{model}
with initial condition $X^{0}$. Then $X$ has finite energy for positive times.
\end{theorem}

\begin{theorem}
\label{teo dissip 2}If $X$ is a positive finite energy solution, then
\[
\lim_{t\rightarrow\infty}\left|  X\left(  t\right)  \right|  _{H}^{2}=0.
\]
Moreover, given $L>0$ and $\alpha>0$, there exists $\bar{t}>0$ depending only
on $L$ and $\alpha$ such that for all positive finite energy solutions $X$
with $|X(0)|_H \leq L$ we have $\left|
X\left(  \bar{t}\right)  \right|  _{H}^{2}\leq\alpha$.
\end{theorem}

The proof of both statements is based on the following lemma

\begin{lemma}
\label{lemma_induz} Let $X$ be a positive componentwise solution, with
$X\left(  0\right)  \in l^{\infty}\cap\mathbb{R}_{+}^{\mathbb{N}}$ and let
$||X(0)||_{\infty}\leq L$. Let $\phi_{n}(t):=\sum_{k=1}^{n}%
X_{k}^{2}(t)$ for $n\geq1$, let $\phi_{0}(t)\equiv0$ and $\phi_{\infty
}(t):=\sum_{k=1}^{\infty}X_{k}^{2}(t)=|X(t)|^2_H$.

Then there exist two summable sequences of positive numbers $\{a_{n}%
\}_{n\geq1}$ and $\{s_{n}\}_{n\geq1}$ depending only on $L$, such that:

\begin{enumerate}[i)]
\item  for all $n\geq1$, for all $t>0$ and for all $\varepsilon\in(0,1]$ one
has
\begin{equation}
\phi_{n}(t+\varepsilon^{-2}s_{n})-\phi_{n-1}(t)\leq\varepsilon a_{n}%
;\label{eq:ts_lemma_induz}%
\end{equation}

\item  for all integers $M\geq1$ one has
\begin{equation}
\phi_{\infty}(\textstyle\varepsilon^{-2}\sum_{k=M}^{\infty}s_{k})\leq
\phi_{M-1}(0)+\varepsilon\sum_{n=M}^{\infty}a_{n}.\label{eq:ts_lemma_induz_2}%
\end{equation}
\end{enumerate}
\end{lemma}

\begin{proof}
For all $n$, let $m_{n}^{2}:=\phi_{n}(0)\vee L$, so that by
equation~\eqref{phi monotona}, $X_{n}(t)\leq m_{n}$ for all $t\geq0$.

We observe that it is possible to find two summable sequences $\{a_{n}\}_n$ and
$\{s_{n}\}_n$ such that for all $n\geq1$:
\[
2^{n}s_{n}a_{n}^{2}\geq2m_{n}^{2}m_{n+2}\left(  1+\frac{1}{2^{n}s_{n}m_{n+2}%
}\right)  .
\]
It is enough to set $s_{n}=2^{-n/4}$ and $a_{n}=C2^{-n/4}$ with $C$ a suitable
constant, and recall that by hypothesis $m_{n}\leq L\sqrt{n}$. We observe that
without changing the sequences, \emph{a fortiori} for all $\varepsilon\leq1$,
\begin{equation}
2^{n}s_{n}a_{n}^{2}\geq2m_{n}^{2}m_{n+2}\left(  1+\frac{1}{\varepsilon
^{-2}2^{n}s_{n}m_{n+2}}\right)  .\label{eq:cond_succ_a_s}%
\end{equation}
\emph{Part 1.} Given the two sequences, we now show that the upper
bounds~\eqref{eq:ts_lemma_induz} hold.

Let $h=\varepsilon^{-2}s_{n}$. Since $\phi$ is nonincreasing, for all
$s\in\lbrack0,h]$,
\[
\phi_{n}(t+h)-\phi_{n-1}(t)\leq\phi_{n}(t+s)-\phi_{n-1}(t+s)=X_{n}^{2}(t+s),
\]
hence, if for some $s$ inside the interval, $X_{n}^{2}(t+s)\leq\varepsilon
a_{n}$ we are done.

On the other hand, let us suppose that $X_{n}^{2}(t+s)>\varepsilon a_{n}$ for
all $0\leq s\leq h$. One has
\[
\phi_{n}(t+h)-\phi_{n-1}(t)=X_{n}^{2}(t)+\int_{0}^{h}\frac{d}{ds}\phi
_{n}(t+s)ds.
\]
We need to prove that
\[
\int_{0}^{h}\frac{d}{ds}\phi_{n}(t+s)ds\leq\varepsilon a_{n}-X_{n}^{2}(t).
\]
For sake of notation simplicity, let $a=\varepsilon a_{n}k_{n}$ and
let $\lambda=k_{n+1}m_{n+2}$. Then, by Theorem~\ref{teo existence 1}-\emph i:
\[
\int_{0}^{h}\frac{d}{ds}\phi_{n}(t+s)ds=\int_{0}^{h}-k_{n}X_{n}^{2}%
(t+s)X_{n+1}(t+s)ds
%
%
\leq-a\int_{0}^{h}X_{n+1}(t+s)ds.
\]
We need a lower bound for $X_{n+1}$. In the interval $[t;t+h]$ we know that
\[
\dot{X}_{n+1}=k_{n}X_{n}^{2}-k_{n+1}X_{n+1}X_{n+2}
%
%
\geq a-\lambda X_{n+1},
\]
whence we get
\[
X_{n+1}(t+s)
%
%
\geq e^{-\lambda s}X_{n+1}(t)+\frac{a}{\lambda}\left(  1-e^{-\lambda
s}\right)  \geq\frac{a}{\lambda}\left(  1-e^{-\lambda s}\right)  .
\]
By substituting into the integral above one gets
\[
\int_{0}^{h}\frac{d}{ds}\phi_{n}(t+s)ds\leq-\frac{a^{2}}{\lambda}\int_{0}%
^{h}\left(  1-e^{-\lambda s}\right)  ds=-\frac{a^{2}}{\lambda^{2}}(e^{-\lambda
h}-1+\lambda h)\leq-\frac{a^{2}}{\lambda^{2}}\frac{\lambda h}{1+\frac
{2}{\lambda h}}.
\]
Substituting next $h$, $a$ and $\lambda$ and recalling the
condition~\eqref{eq:cond_succ_a_s}, we finally get
\[
\int_{0}^{h}\frac{d}{ds}\phi_{n}(t+s)ds\leq-\frac{a^{2}h}{\lambda(1+\frac
{2}{\lambda h})}=-\frac{2^{n}s_{n}a_{n}^{2}}{2m_{n+2}\left(  1+\frac
{1}{\varepsilon^{-2}2^{n}s_{n}m_{n+2}}\right)  }\leq-m_{n}^{2}\leq\varepsilon
a_{n}-X_{n}^{2}(t).
\]
This concludes the first part.

\emph{Part 2.} Let $M\geq1$ and define the sequence $\{t_{n}\}_{n\geq M-1}$ by
$t_{M-1}=0$ and $t_{n}=\varepsilon^{-2}\sum_{k=M}^{n}s_{k}$, for $n\geq M$. We
substitute $t=t_{n-1}$ inside the inequality~\eqref{eq:ts_lemma_induz}:
\[
\phi_{n}(t_{n})-\phi_{n-1}(t_{n-1})=\phi_{n}(t_{n-1}+\varepsilon^{-2}%
s_{n})-\phi_{n-1}(t_{n-1})\leq\varepsilon a_{n}.
\]
Adding the above inequalities for $n$ from $M$ to any number $N>M$, one has
\[
\phi_{N}(t_{N})-\phi_{M-1}(0)=\sum_{n=M}^{N}\left[  \phi_{n}(t_{n})-\phi
_{n-1}(t_{n-1})\right]  \leq\varepsilon\sum_{n=M}^{N}a_{n}.
\]
Monotonicity of $\phi_{N}$ yields that
\[
\phi_{N}(\varepsilon^{-2}{\sum_{k=M}^{\infty}}s_{k})\leq\phi_{N}(t_{N}%
)\leq\phi_{M-1}(0)+\varepsilon\sum_{n=M}^{N}a_{n},
\]
hence letting $N$ go to infinity we get~\eqref{eq:ts_lemma_induz_2}. The proof
of the lemma is complete.
\end{proof}

We are now ready to prove the above theorems. Consider the assumptions
of Theorem~\ref{teo dissip 1}. By letting $\varepsilon=1$ in the
second part of Lemma~\ref{lemma_induz}, one has that $X(t)\in H$ for
$t\geq\sum_{k=M} ^{\infty}s_{k}$. Letting $M$ go to infinity we get
the thesis of Theorem~\ref{teo dissip 1}.

As to Theorem~\ref{teo dissip 2}, let us prove that for all $\alpha>0$
there exists some $\bar{t}>0$ such that
$\phi_{\infty}(\bar{t})\leq\alpha$.  Since
$||X(0)||_\infty\leq|X(0)|_H\leq L$, we let $M=1$ in the second part
of Lemma~\ref{lemma_induz}: we only need to choose $\varepsilon$ in
such a way that $\varepsilon\sum_{n=1}^{\infty} a_{n}\leq\alpha$. One
gets $\bar{t}=\varepsilon^{-2}\sum_{k=1}^{\infty}s_{k}$.  The proof is
complete.

\section{Bounds on the decay of energy as $t\rightarrow\infty$}

In this section we prove a bound from above and another from below, for the
decay of energy as $t\rightarrow\infty$, which essentially say that solutions
decay as $t^{-1}$. The results are restricted to \textit{positive}
componentwise solutions. The first result is due to a scaling argument based
on the fact that the nonlinearity is homogeneous of degree two.

\begin{theorem}
Let $X$ be a positive componentwise solution, with $X\left(  0\right)  \in
l^{\infty}\cap\mathbb{R}_{+}^{\mathbb{N}}$. Then there exists $C>0$ such that
\[
\left|  X\left(  t\right)  \right|  _{H}^{2}\leq\frac{C}{t^{2}}%
\]
for $t\geq1$.
\end{theorem}

\begin{proof}
First, by Theorem~\ref{teo dissip 1}, the solution has finite energy for
positive times. Hence, by Theorem~\ref{teo dissip 2}, there is a time $t_{0}$
such that $\left|  X\left(  t_{0}\right)  \right|  _{H}^{2}\leq1$. It is thus
not restrictive to prove the theorem in the case $\left|  X^{0}\right|
_{H}^{2}=1$.

First a general fact: from Theorem~\ref{teo dissip 2} we know that there
exists $\bar{t}>0$ such that for all initial conditions $X^{0}$ with $\left|
X^{0}\right|  _{H}^{2}=1$ we have
\[
\left|  X\left(  \bar{t}\right)  \right|  _{H}^{2}\leq\frac{1}{4}.
\]

Let us start the proof, for a solution $X$ such that $\left|  X^{0}\right|
_{H}^{2}=1$. Let $\alpha_{1}^{2}=\left|  X\left(  \bar{t}\right)  \right|
_{H}^{2}$. Consider the rescaled $H$-valued function $Y=\left(  Y_{n}\right)
_{n\in\mathbb{N}}$ defined as
\[
Y(t):=\alpha_{1}^{-1}X\left(  \alpha_{1}^{-1}t+\bar{t}\right)  .
\]
We have
\[
\dot{Y}_{n}(t)=\alpha_{1}^{-2}k_{n-1}X_{n-1}^{2}\left(  \alpha_{1}^{-1}%
t+\bar{t}\right)  -\alpha_{1}^{-2}k_{n}X_{n}\left(  \alpha_{1}^{-1}t+\bar
{t}\right)  X_{n+1}\left(  \alpha_{1}^{-1}t+\bar{t}\right)
\]%
\[
=k_{n-1}Y_{n-1}^{2}\left(  t\right)  -k_{n}Y_{n}\left(  t\right)
Y_{n+1}\left(  t\right)  .
\]
This means that $Y$ is a finite energy solution of system~\eqref{model}.
Moreover, $\left|  Y\left(  0\right)  \right|  _{H}^{2}=1$. Hence
\[
\alpha_{2}^{2}:=\left|  Y\left(  \bar{t}\right)  \right|  _{H}^{2}\leq\frac
{1}{4}.
\]
In terms of $X$ we have
\[
\left|  X\left(  \frac{\bar{t}}{\alpha_{1}}+\bar{t}\right)  \right|  _{H}%
^{2}=\alpha_{1}^{2}\alpha_{2}^{2}.
\]
By induction we can prove that
\[
\left|  X\left(  \frac{\bar{t}}{\alpha_{1}\cdot\alpha_{2}\cdot\ldots
\cdot\alpha_{k}}+\dots+\frac{\bar{t}}{\alpha_{1}}+\bar{t}\right)  \right|
_{H}^{2}=\alpha_{1}^{2}\cdot\alpha_{2}^{2}\cdot\ldots\cdot\alpha_{k+1}^{2}%
\]
for every $k\geq1$. Since $\alpha_{i}\leq\frac{1}{2}$, we have
\[
\frac{\bar{t}}{\alpha_{1}\cdot\alpha_{2}\cdot\ldots\cdot\alpha_{k}}%
+\dots+\frac{\bar{t}}{\alpha_{1}}+\bar{t}<\frac{2\bar{t}}{\alpha_{1}%
\cdot\alpha_{2}\cdot\ldots\cdot\alpha_{k}}.
\]
Recall that energy inequality holds for all positive finite energy solutions,
see Theorem~\ref{teo existence 1}. Hence for all $t\geq\frac{2\bar{t}}%
{\alpha_{1}\cdot\alpha_{2}\cdot\ldots\cdot\alpha_{k}}$ we have
\[
\left|  X\left(  t\right)  \right|  _{H}^{2}\leq\alpha_{1}^{2}\cdot\alpha
_{2}^{2}\cdot\ldots\cdot\alpha_{k+1}^{2}.
\]
If we restrict to $\frac{2\bar{t}}{\alpha_{1}\cdot\alpha_{2}\cdot\ldots
\cdot\alpha_{k}}\leq t\leq\frac{2\bar{t}}{\alpha_{1}\cdot\alpha_{2}\cdot
\ldots\cdot\alpha_{k}\cdot\alpha_{k+1}}$ we also have $\alpha_{1}^{2}%
\cdot\alpha_{2}^{2}\cdot\ldots\cdot\alpha_{k+1}^{2}\leq\frac{4\bar{t}^{2}%
}{t^{2}}$. Hence
\[
\left|  X\left(  t\right)  \right|  _{H}^{2}\leq\frac{4\bar{t}^{2}}{t^{2}}%
\]
for all $\frac{2\bar{t}}{\alpha_{1}\cdot\alpha_{2}\cdot\ldots\cdot\alpha_{k}%
}\leq t\leq\frac{2\bar{t}}{\alpha_{1}\cdot\alpha_{2}\cdot\ldots\cdot\alpha
_{k}\cdot\alpha_{k+1}}$. This implies the claim of the theorem. The proof is complete.
\end{proof}

\begin{theorem}
Let $X$ be a positive componentwise solution, with $X\left(  0\right)  \in
l^{\infty}\cap\mathbb{R}_{+}^{\mathbb{N}}$. Let $n_{0}+1$ be the minimum
integer with the property $X_{n_{0}+1}^{0}>0$. We know from 
Theorem~\ref{teo existence 1} that $X_{n}(t)>0$ for all $n> n_{0}$ and
$t>0$. Then, for some constant $C>0$, for every $n> n_{0}$ and
$t\geq1$ we have
\[
\int_{1}^{t}X_{n+1}(s)ds\geq k_{n}^{-1}\log t+k_{n}^{-1}\log\left(
\frac{X_{n}(1)}{C}\right)  .
\]
Thus, in particular, for every $n> n_{0}$,
\[
\int_{1}^{\infty}X_{n+1}(s)ds=\infty
\]
and
\[
\underset{t\rightarrow\infty}{\lim\sup}\, t\cdot X_{n+1}(t)\geq k_{n}^{-1}.
\]
\end{theorem}

\begin{proof}
From identity~\eqref{variation of constants} we have
\[
X_{n}(t)\geq X_{n}(1)e^{-k_{n}\int_{1}^{t}X_{n+1}(s)ds}.
\]
By the upper bound of the previous theorem, there exists a constant $C>0$ such
that
\[
X_{n}(1)e^{-k_{n}\int_{1}^{t}X_{n+1}(s)ds}
\leq X_n(t)
\leq\frac{C}{t}%
\]
hence
\[
-k_{n}\int_{1}^{t}X_{n+1}(s)ds\leq\log\left(  \frac{C}{X_{n}(1)\cdot
t}\right)
\]
namely
\[
\int_{1}^{t}X_{n+1}(s)ds\geq\frac{\log\left(  \frac{X_{n}(1)}{C}\right)  +\log
t}{k_{n}}.
\]
This implies the claims of the theorem. The proof is complete.
\end{proof}

\section{Self-similar solutions and related facts}

We call \textit{self-similar} any solution $X$ of the form $X_{n}%
(t)=a_{n}\cdot\varphi(t)$. It is easy to verify that self-similar solutions
satisfying the equations~\eqref{model} have the form
\begin{equation}
X_{n}(t)=\frac{a_{n}}{t-t_{0}},\quad t> t_{0}, \label{self-similar}%
\end{equation}
with $t_0<0$.
We are interested in \textit{finite energy} self-similar solutions, hence we
require also
\[
\sum_{n=1}^{\infty}a_{n}^{2}<\infty.
\]
In the next section we prove the following result.

\begin{theorem}
\label{teo self-similar 1}Given $t_{0}<0$, there exists a unique finite energy
self-similar solution with $a_{1}\neq0$. In general, given $t_{0}<0$ and
$n_{0}\geq0$, there exists a unique finite energy self-similar solution with
\[
a_{1}=...=a_{n_{0}}=0,\quad a_{n_{0}+1}\neq0
\]
(where the first conditions are meaningful only for $n_{0}>0$). In addition,
the coefficients $a_{n}$ have the property%
\[
\lim_{n\rightarrow\infty}\frac{a_{n}}{k_{n}^{-1/3}}=C_{n_{0}}.
\]
\end{theorem}

Thus we see that Kolmogorov scaling law \cite{K41} (so called K41) appears in
these special solutions, phenomenologically associated to decaying turbulence.
But it is for us a very difficult open problem to understand whether all other
solutions approach the self-similar ones and in which sense, so we cannot say
how general this scaling property should be considered.

The existence of finite energy self-similar solutions is of conceptual
interest in itself, in comparison with analogous investigations for Euler and
Navier-Stokes equations. Also here, in this very simple context, the proof is
highly non trivial. Apart from its intrinsic interest, the existence of such
solutions has a number of implications.

First, they realize perfectly the decay $t^{-1}$, coherently with the previous
section (where the lower bound was more vague). We conjecture that the set of
all finite energy self-similar solutions, set depending on $t_{0}\in
\mathbb{R}$ and $n_{0}\geq0$, attracts all other solutions in a suitable
sense. If this is the case, the decay $t^{-1}$ would be more strictly the true
one for all solutions.

A second implication is the existence of solutions that blow-up backward in
time. This is of interest for two reasons. To explain them let us first
clarify what happens to solutions when we reverse time.

We may consider system~\eqref{model} for negative times, $t\leq0$, and give a
definition of componentwise solution exactly as for $t\geq0$. All definitions
and theorems can be rewritten for negative times. One also has the following
correspondence between the forward and backward problem: let $X=\left(
X\left(  t\right)  \right)  _{t\geq0}$ be a componentwise solution of 
system~\eqref{model}, for $t\geq0$, as usual. Then $Y=\left( Y\left(
t\right) \right) _{t\leq0}$ defined as
\begin{equation}
Y\left(  t\right)  =-X\left(  -t\right)  \label{inversion of time}%
\end{equation}
is a componentwise solution for $t\leq0$. Indeed, by the purely quadratic
nature of the equation,
\begin{align*}
\dot{Y_{n}}(t)  &  =\dot{X_{n}}(-t)=k_{n-1}X_{n-1}^{2}(-t)-k_{n}%
X_{n}(-t)X_{n+1}(-t)\\
&  =k_{n-1}Y_{n-1}^{2}(t)-k_{n}Y_{n}(t)Y_{n+1}(t).
\end{align*}
Thus, any positive solution over $[0,\infty)$ gives rise to a negative
solution over $(-\infty,0]$, and vice versa.

Theorem~\ref{teo self-similar 1} ensures that there exists a self-similar
solution $X_{n}(t):=\frac{a_{n}}{t-t_{0}}$ with $t_{0}<0$ and $a_{n}>0$ for
all $n>n_{0}$. It is easy to check that $X_{n}(t)$ is a componentwise solution
on the open interval $(t_{0},+\infty)$. The energy is finite for all
$t>t_{0}$ and $\lim_{t\to\infty}|X_{n}(t)|=0$ whereas $\lim_{t\to t_{0}^{+}%
}X_{n}(t)=+\infty$ for all $n>n_{0}$. With time 
inversion~\eqref{inversion of time} it is possible to define
$Y(t):=-X(-t)$ for all $n$ and $t\in (-\infty,-t_{0})$. $Y(t)$ is a
componentwise solution of~\eqref{model} on the open interval
$(-\infty,-t_{0})$ and $\lim_{t\to-\infty}|Y_{n}(t)|=0$ whereas
$\lim_{t\to-t_{0}^{-}}|Y_{n}(t)|=+\infty$. This means that $Y_{n}$
blows-up in finite time. Notice that every component blow-up, not only
some norm of the solution.

A consequence of this is a coalescence property of self-similar
solutions. Theorem~\ref{teo existence 2} applied with initial
condition $Y(0)$, states that there exists a finite energy solution
$\tilde Y(t)$, defined for all $t\geq0$, with bounded energy and
initial condition $\tilde Y(0)=Y(0)$. Moreover, the 
condition~\eqref{energy inequality} of the theorem ensure that the
energy of $\tilde Y$ is nonincreasing in $t$, on the contrary $|Y|_H$
is increasing in time. This means that on the system~\eqref{model}
\emph{there is not uniqueness of solutions in
$\mathbb{R}^{\mathbb{N}}$}.  This result is of interest in itself, but
we stress that it holds in the enlarged class of non-necessarily
positive solutions.

Let us invert the time again with $\tilde{X}_{N}(t):=-\tilde
Y_{N}(-t)$. $X$ and $\tilde{X}$ are two componentwise solutions on the
interval $(t_{0},0]$, with decreasing in time and nondecreasing in
time energy respectively. They are different on $(t_0,0)$ and coincide
at time $t=0$. The role of time $t=0$ in the previous argument can be
replaced by any time $t_1>t_0$.  Thus we have

\begin{corollary}
If $X$ is a self-similar solution of the form~\eqref{self-similar}, all its
values are coalescence points, in the sense that for all $t_{1}\geq t_{0}$
there exists a finite energy solution $X^{t_{1}}$, defined for $t\in(-\infty, t_1]$,
such that $X^{t_{1}}\left(  t_{1}\right)  =X\left(  t_{1}\right)  $ and
$X^{t_{1}}(t)\neq X(t)$ on $(t_{0},t_{1})$.
\end{corollary}

Finally, we state the blow-up for negative self-similar finite energy
solutions.

\begin{corollary}
Given $t_{0}>0$ and $n_{0}\geq0$, there exists a unique negative
finite energy self-similar solution, defined on $[0,t_0)$ with
\[
a_{1}=...=a_{n_{0}}=0,\quad a_{n_{0}+1}\neq0.
\]
This solution blows-up at time $t_0$.
\end{corollary}
\begin{proof}
Apply the inversion~\eqref{inversion of time} to the solution given by
Theorem~\ref{teo self-similar 1}.
\end{proof}

\section{Existence and uniqueness of self-similar solutions}

In order to prove Theorem~\ref{teo self-similar 1} it is best to restate the
problem in terms of properties for sequences of positive numbers.

If a positive componentwise solution is of the form~\eqref{self-similar}, then
\[
-\frac{a_{n}}{(t-t_{0})^{2}} =\dot X_{n}(t) =k_{n-1}X_{n-1}^{2}(t)-k_{n}%
X_{n}(t)X_{n+1}(t) =k_{n-1}\frac{a_{n-1}^{2}}{(t-t_{0})^{2}}-k_{n}\frac
{a_{n}a_{n+1}}{(t-t_{0})^{2}}
\]
so the sequence $\{a_{n}\}$ must satisfy
\[
a_{n}a_{n+1}=2^{-n}a_{n}+a_{n-1}^{2}/2
\]
for all $n$. It is indeed possible for the first terms $a_{1},a_{2}%
,\dots,a_{n_{0}}$ to be zero, but by induction, if $a_{n_{0}+1}>0$ then the
subsequent coefficients must satisfy
\begin{equation}\label{eq:ss_coeff_rec}
a_{n+1}=2^{-n}+\frac{a_{n-1}^{2}}{2a_{n}}>0,\qquad n\geq n_{0}+1
\end{equation}
To simplify matters and take into account $n_{0}$, let
\[
\widetilde a_{n}:=2^{n+n_{0}}a_{n+n_{0}}
\]
so that the condition becomes $\widetilde a_{0}=0$, $\widetilde a_{1}>0$ and
\begin{equation}\label{eq:ss_coeff_tilde_rec}
\widetilde a_{n+1}=2+4\frac{\widetilde a_{n-1}^{2}}{\widetilde a_{n}},\qquad
n\geq1
\end{equation}
Since, given $a_{n_{0}+1}$ or $\widetilde a_{1}$, this recurrence uniquely
defines the sequence, in order to prove Theorem~\ref{teo self-similar 1} we
have to show that there exists a unique positive number $\widetilde a_{1}$
such that the sequence $\{a_{n}\}_n\in H$. We will prove the following.

\begin{theorem}\label{thm:ss_2} 
There exists a unique real number $\gamma$ such that the sequence
$\{a_{n}\}_n$ is in $H$ iff $\widetilde{a}_{1}=\gamma$ (equivalently,
iff $a_0=a_1=\dots=a_{n_0}=0$ and $a_{n_0+1}=2^{-n_0-1}\gamma$).

Moreover, let $\beta=2^{-1/3}$. One can find $R>0$ and a strictly
decreasing bijective function $h:(0;R]\rightarrow\lbrack0;\infty)$
such that
\begin{enumerate}[i)]
\item $\gamma=h(\beta^{2}R)$,
\item if $\{a_n\}_n\in H$, then $a_{n}=2^{-n}h(\beta^{2(n-n_{0})}R)$ for all
$n>n_0$. In particular $a_{n}\sim C_{n_{0} }\beta^{n}$, for
$n\rightarrow\infty$, with $C_{n_{0}}=\beta^{2n_{0}}/R$.
\end{enumerate}
\end{theorem}

Numerical computations give $\gamma\approx0.917576296$.

It is clear that Theorem~\ref{teo self-similar 1} follows immediately from
Theorem~\ref{thm:ss_2}. The proof of the latter requires some work and will
follow from Theorem~\ref{thm:ss_3} below. Three technical lemmas will be needed.

We start by looking for another sequence $\{d_k\}_{k\geq-1}$ (not
depending on $a_n$), satisfying the peculiar relation below:
\begin{equation}\label{eq:recur_d}
\begin{cases}
d_{-1}=-1 \\
\displaystyle\sum_{i=-1}^{k+1}d_{i}d_{k-i}(4-\beta^{2k+2i})=2d_{k}\beta^{2k}, & \qquad k\geq-1
\end{cases}
\end{equation}

\begin{lemma}
Equations~\eqref{eq:recur_d} uniquely define a sequence of real
numbers $\{d_k\}_{k\geq-1}$ such that $d_k\geq0$ for $k\geq0$. The
power series $\sum_{k=0}^\infty d_kx^k$ has a positive convergence
radius $R>0$. If we let $h(x):=-\sum_{k=-1}^\infty d_kx^k$, the
function $h$ is defined on the interval $(0;R)$ where it is analytic,
nonnegative and strictly decreasing, with $h(0^+)=+\infty$ and
$h(R^-)=0$ (so that it can be continuously extended on $(0;R]$).
\end{lemma}

\begin{proof}
The second one of~\eqref{eq:recur_d} uniquely defines $d_{k+1}$ as a
function of the previous terms. Truly, the coefficient of $d_{k+1}$ is
$d_{-1}(8-\beta^{2k-2}-\beta^{4k+2})$ which is always nonzero.

With some algebraic manipulations, the above system of equations can
be rewritten as:
\begin{equation}\label{eq:recur_d_bis}
\begin{cases}
d_{-1}=-1 \\
d_{0}=(2\beta^{-1}-\beta-\beta^3)^{-1}\approx0.8155665 \\
d_{k+1}=\frac12\sum_{i=0}^{k}\alpha_{k,i}d_{i}d_{k-i}, &\qquad k\geq0,
\end{cases}
\end{equation}
where
\[
\alpha_{k,i}=\begin{cases}
\dfrac{1+\beta^4-\beta^{-1}}{1-\beta^{7}-\beta^{11}} & k=0,\quad i=0\\[2ex]
\dfrac{(1-\beta^{2k+2})(1+\beta^{2k+7})}{1-\beta^{2k+7}-\beta^{4k+11}} & k>0,\quad i=0,k \\[2ex]
\dfrac{1-\beta^{4k-2i+9}-\beta^{2k+2i+9}}{1-\beta^{2k+7}-\beta^{4k+11}} & 0<i<k\end{cases}
\]
All $\alpha_{k,i}$ are positive and so $d_k\geq0$ for $k\geq0$. We
notice moreover that $\lim_{k\rightarrow\infty} \alpha_{k,i}=1$
uniformly in $i$ (exponentially in $k$, see the Appendix).

The recursion~\eqref{eq:recur_d_bis} is similar to the classical
Catalan sequence, but since we have only some asymptotic control on
the coefficients $\alpha_{k,i}$, and since the behaviour of
$\{d_k\}_k$ strongly depends on the first values (even its convergence
radius does), we need some explicit investigation of its properties.

For $k\geq0$, let
$d'_{k+1}:=\frac12\sum_{i=0}^{k}(\alpha_{k,i}-1)d_{i}d_{k-i}$ and let
$d'_0=d_0$.

Let $g(x):=\sum_{k=0}^\infty d_kx^k$ so that $h(x)=x^{-1}-g(x)$. Let
$\hat g(x):=\sum_{k=0}^\infty d'_kx^k$. By the third one
of~\eqref{eq:recur_d_bis} it follows that
\[
\sum_{k=n}^\infty d_{k+1}x^{k+1}=x\frac12\sum_{k=n}^\infty \sum_{i=0}^{k}\alpha_{k,i}d_{i}x^id_{k-i}x^{k-i}
\]
If $n$ is such that $|\alpha_{k,i}-1|<\epsilon$, we have
\[
g(x)-\sum_{k=0}^{n}d'_kx^k
\leq x\frac{1+\epsilon}2\sum_{k=0}^\infty \sum_{i=0}^{k}d_{i}x^id_{k-i}x^{k-i}
=\frac{1+\epsilon}2xg^2(x)
\]
Writing also the corresponding lower bound and letting
$\epsilon\rightarrow0$ and $n\rightarrow\infty$ accordingly, yields
\[
g(x)-\hat g(x)=\frac12xg^2(x).
\]
The above formula is true for all (complex) $x$ inside the radius of
convergence of $g$ (in the Appendix we show that $\hat g$ has a
convergence radius $\beta^{-2}$ times larger than $g$).

For all $x$ inside the convergence radius of $g$, $g(x)$ must be one
of the roots of the above degree 2 polynomial, i.e.
\begin{equation}\label{eq:algeb_g}
g(x)
=\frac{1-\sqrt{1-2x\hat g(x)}}x
\end{equation}
(The other root does not satisfy $g(0)=d_0$.)

In the Appendix we show that the radius of convergence of $\hat g$ is
greater than 1 and that inside the unit circle of the complex plane
$1-2z\hat g(z)$ has only one zero, which is in fact some real point
$R\in(4/5, 1)$. It follows that $R$ is the radius of convergence of
$g$, that $h$ is defined on $(0;R)$ and that $h(R^-)=0$. The other
properties follow from inspection of the first coefficients $d_k$.
\end{proof}

From now on $h$ is intended to be extended up to $R$. Since $h$ is
bijective and has image $[0,\infty)$ and recalling that $\widetilde
a_n>0$ for all $n$, it makes sense to compute $h^{-1}(\widetilde a_n)$.

\begin{lemma}
Let $\{\widetilde a_n\}_{n\geq0}$ be any positive sequence
satisfying~\eqref{eq:ss_coeff_tilde_rec}, not necessarily with
$\widetilde a_0=0$. For $n\geq0$ let $\lambda_n:=h^{-1}(\widetilde
a_n)\in(0,R]$. The numbers $\lambda_n$ satisfy:
\begin{equation}\label{eq:recur_lambda}
h(\lambda_{n+1})
=2\left(1-\frac{h(\beta^2\lambda_{n-1})}{h(\lambda_n)}\right)+\frac{h(\beta^4\lambda_{n-1})h(\beta^2\lambda_{n-1})}{h(\lambda_n)},\qquad\qquad n\geq 1
\end{equation}
\end{lemma}

\begin{proof}
Recall that $h(x)=-\sum_{k=-1}^\infty d_kx^k$,
\[
\widetilde a^2_{n-1}
=h^2(\lambda_{n-1})
=\sum_{i,j=-1}^\infty d_id_j\lambda_{n-1}^i\lambda_{n-1}^j
=\sum_{k=-2}^\infty\lambda_{n-1}^k\sum_{i=-1}^{k+1} d_id_{k-i}
\]
Using relation~\eqref{eq:recur_d} for $k\geq-1$ and since
$\beta^{2k+2i}\equiv4$ if $k=-2$ we get:
\begin{multline*}
4\widetilde a^2_{n-1}
=\sum_{k=-1}^\infty 2d_k\beta^{2k}\lambda_{n-1}^k+\sum_{k=-2}^\infty\lambda_{n-1}^k\sum_{i=-1}^{k+1} \beta^{2k+2i}d_id_{k-i}\\
=2\sum_{k=-1}^\infty d_k\beta^{2k}\lambda_{n-1}^k+\sum_{i,j=-1}^\infty\beta^{4i+2j} d_id_j\lambda_{n-1}^{i+j}
=-2h(\beta^2\lambda_{n-1})+h(\beta^4\lambda_{n-1})h(\beta^2\lambda_{n-1})
\end{multline*}
Substituting in equation~\eqref{eq:ss_coeff_tilde_rec} we find
\[
h(\lambda_{n+1})
=\widetilde a_{n+1}
=2+4\frac{\widetilde a^2_{n-1}}{\widetilde a_n}
=2+\frac{-2h(\beta^2\lambda_{n-1})+h(\beta^4\lambda_{n-1})h(\beta^2\lambda_{n-1})}{h(\lambda_n)}
\]
which is equivalent to~\eqref{eq:recur_lambda}.
\end{proof}

The ratio between $\lambda_n$ and $\beta^2\lambda_{n-1}$ plays a
crucial r\^ole in what follows. Next lemma characterizes its logarithm.

\begin{lemma}\label{lem:lambda_prime}
Let $\{\widetilde a_n\}_{n\geq0}$ be any positive sequence
satisfying~\eqref{eq:ss_coeff_tilde_rec}, not necessarily with
$\widetilde a_0=0$. For $n\geq1$ let
$\lambda'_n:=\log\lambda_n-\log\lambda_{n-1}-\log\beta^2$. The
sequence $\lambda'_n$ can present two behaviours only:
\begin{enumerate}
\item If for any $n\geq1$, $\lambda'_n=0$, then $\lambda'_m=0$ for all
$m\geq1$.
\item If $\lambda'_n\neq0$ for some $n\geq1$, then there exist two
constants $c,d>0$ such that for all $k\geq0$:
\[
|\lambda'_{n+k}|\geq c(1+d)^{k}|\lambda'_n|,\qquad\qquad
(-1)^k\lambda'_{n+k}\lambda'_n>0
\]
\end{enumerate}
\end{lemma}

\begin{proof}
If $\lambda'_n=0$, then of course $\lambda_n=\beta^2\lambda_{n-1}$,
hence~\eqref{eq:recur_lambda} reduces to:
\[
h(\lambda_{n+1})
=h(\beta^4\lambda_{n-1})
=h(\beta^2\lambda_{n})
\]
Since $h$ is injective, $\lambda_{n+1}=\beta^2\lambda_n$, that is
$\lambda'_{n+1}=0$.  This proves by induction the first part of the
lemma, for $k\geq n$. If the same were not true for some $k<n$, then
by the second part it would have been $\lambda'_n\neq0$.

In particular note that if $\widetilde a_0=0$ (yielding $\lambda_0=R$)
and $\lambda_1=\beta^2R$ (yielding $\lambda'_1=0$), then by the first
part and recurrence~\eqref{eq:ss_coeff_tilde_rec}, $h(\beta^4R)=\widetilde
a_2=2$. This will be used in the second part.

\medskip

For proving the second part, let $\psi(x):=\log h(e^x)$ and let
$\Delta_x(y):=h(xe^y)/h(x)$. 

The function $\psi$ is defined on $(-\infty, \log R)$, where it is
analytic, strictly decreasing, concave and invertible; $\psi$ has one
oblique and one vertical asymptote: $\psi(x)+x\uparrow0$ for
$x\downarrow-\infty$ and $\psi(x)\sim\frac12\log(\log R-x)$ for
$x\uparrow\log R$, so that the image of $\psi$ is the whole $\mathbb
R$. Moreover $\psi'<-1$ on all the domain.

$\Delta_x(y)$ on the other hand is defined for $x\in(0; R)$ and
$y\in(-\infty;\log R-\log x)$ and this last interval always contains a
neighbourhood of 0. $\Delta_x(\cdot)$ is analytic and decreasing on
all its domain, moreover we have:
\begin{align*}
\Delta_x(0)&=1\\
\Delta_x(y)&> e^{-y}>1,\qquad y<0\\
\Delta_x(y)&< e^{-y}<1,\qquad y>0
\end{align*}
in fact,
\[
\Delta_x(y)=h(xe^{y})/h(x)
=\exp(\psi(\log x+y)-\psi(\log x))
=\exp(\psi'(\xi)y)
\]
and $\psi'(\xi)<-1$.
By virtue of~\eqref{eq:recur_lambda} we can write:
\begin{multline*}
\Delta_{\beta^2\lambda_n}(\lambda'_{n+1})
=\frac{h(\lambda_{n+1})}{h(\beta^2\lambda_n)}\\
=\frac2{h(\beta^4\lambda_{n-1})}\frac{h(\beta^4\lambda_{n-1})}{h(\beta^2\lambda_{n})}\left(1-\frac{h(\beta^2\lambda_{n-1})}{h(\lambda_n)}\right)+\frac{h(\beta^4\lambda_{n-1})h(\beta^2\lambda_{n-1})}{h(\beta^2\lambda_n)h(\lambda_n)}\\
=\frac2{h(\beta^4\lambda_{n-1})}\Delta_{\beta^2\lambda_{n}}(-\lambda'_n)\left(1-\Delta_{\lambda_{n}}(-\lambda'_n)\right)+\Delta_{\beta^2\lambda_{n}}(-\lambda'_n)\Delta_{\lambda_{n}}(-\lambda'_n)\\
=2\left[\frac12-\frac1{h(\beta^4\lambda_{n-1})}\right]\Delta_{\beta^2\lambda_{n}}(-\lambda'_n)\Delta_{\lambda_{n}}(-\lambda'_n)+\frac2{h(\beta^4\lambda_{n-1})}\Delta_{\beta^2\lambda_{n}}(-\lambda'_n)
\end{multline*}
Let
\[
K_n:=\frac12-\frac1{h(\beta^4\lambda_{n-1})}\leq\frac12
\]
We need to prove that $K_n$ is always nonnegative, and bounded below
by a positive constant for $n$ large enough.

Since $\lambda_{n-1}\leq R$, $h(\beta^4\lambda_{n-1})\geq
h(\beta^4R)=2$. So $K_n\geq0$ for all $n\geq1$.  For $n\geq3$,
$\widetilde a_{n-1}\geq2$ by~\eqref{eq:ss_coeff_tilde_rec}, hence
$\lambda_{n-1}\leq\beta^4R$, so:
\[
K_n
\geq\frac12-\frac1{h(\beta^4h^{-1}(2))}
=\frac12-\frac1{h(\beta^8R)}
=:K>0.
\]
Going back to the previous expression, if $\lambda'_n>0$, since both
coefficients are positive,
\begin{align*}
\Delta_{\beta^2\lambda_n}(\lambda'_{n+1})
&=2K_n\Delta_{\beta^2\lambda_{n}}(-\lambda'_n)\Delta_{\lambda_{n}}(-\lambda'_n)+(1-2K_n)\Delta_{\beta^2\lambda_{n}}(-\lambda'_n)\\
&>2K_ne^{2\lambda'_n}+(1-2K_n)e^{\lambda'_n}\\
&=2K_ne^{2\lambda'_n}+(1-2K_n)e^{\lambda'_n}-e^{\lambda'_n}+e^{\lambda'_n}-1+1\\
&=1+(2K_ne^{\lambda'_n}+1)(e^{\lambda'_n}-1)
>1
\end{align*}

\emph{A posteriori} we get $\lambda'_{n+1}<0$ and so, after letting
$x:=e^{\lambda'_n}$,
\[
e^{-\lambda'_{n+1}}
>\Delta_{\beta^2\lambda_n}(\lambda'_{n+1})
>2K_nx^2+(1-2K_n)x
=:K_nx(x-1)+\theta_n(x)
\geq \theta_n(x),
\]
where
\[
\theta_n(x):=K_n x^2+(1-K_n)x\geq x^{1+K_n}=:\widetilde\theta_n(x)
\]
The above inequality holds for all $x\geq1$ and follows by comparing
the two functions, their derivatives and by the fact that $K_n<1$,
in fact:
\begin{gather*}
\theta_n(1)=1=\widetilde\theta_n(1)
\qquad\qquad
\theta'_n(1)=1+K_n=\widetilde\theta'_n(1)\\
\theta''_n(x)=2K_n\geq K_n+K_n^2\geq(K_n+K_n^2)x^{(K_n-1)}=\widetilde\theta_n''(x),\quad\text{for }x\geq1
\end{gather*}
Finally $e^{-\lambda'_{n+1}} > \widetilde\theta_n(x)=e^{(1+K_n)\lambda'_n}$
and hence
\[
\lambda'_n>0
\quad\Rightarrow\quad
\lambda'_{n+1} < -(1+K_n)\lambda'_n
\]
On the other hand, if $\lambda'_n<0$, with analogous reasoning:
\[
e^{-\lambda'_{n+1}}
<\Delta_{\beta^2\lambda_n}(\lambda'_{n+1})
<2K_nx^2+(1-2K_n)x
=2K_nx(x-1)+x<x=e^{\lambda'_n}
\]
and hence
\[
\lambda'_n<0
\quad\Rightarrow\quad
\lambda'_{n+1} > -\lambda'_n
\]
Summing up things, for any $\lambda'_n\neq0$, the sequence
$\{\lambda'_{n+k}\}_k$ has alternating signs, moreover:
\[
|\lambda'_{n+k+2}|\geq(1+K_{n+k})|\lambda'_{n+k}|
\]
The second part of the lemma is now proved by induction on $k$,
recalling that $K_n\geq K>0$ for $n\geq3$.
\end{proof}

We are finally able to make the main statement.
\begin{theorem}\label{thm:ss_3}
§§§A sequence $a_n$ satisfying~\eqref{eq:ss_coeff_rec} with $a_{n_0}=0$
and $a_{n_0+1}>0$ can present only two behaviours:
\begin{enumerate}
\item If for any $m\geq1$, $\lambda'_{m}=0$, then there exists $c>0$
such that for all $n> n_0$, $a_n=2^{-n}h(\beta^{2(n-n_0)}R)$; in
this case $a_n\rightarrow0$ for $n\uparrow\infty$ and in the limit
$a_n=C_{n_0}\beta^n+O(\beta^{3n})$, with $C_{n_0}=\beta^{2n_0}/R$.
\item If for any $m\geq1$, $\lambda'_{m}\neq0$, then either the
odd or the even terms of $a_n$ diverge more than exponentially and the
power series $\sum_na_nz^n$ has zero radius of convergence.
\end{enumerate}
\end{theorem}
\begin{proof}
We apply directly Lemma~\ref{lem:lambda_prime}. 

In the first case, for all $n\geq1$, $\lambda'_n\equiv0$, hence
$\lambda_{n}=\beta^{2n}\lambda_0=\beta^{2n}R$, by which
$a_{n+n_0}=2^{-n_0-n}\widetilde a_n=2^{-n_0-n}h(\beta^{2n}R)$. The
limit behaviour of $a_n$ follows from $h(x)=x^{-1}+O(1)$ for
$x\downarrow0$.

\medskip

In the second case, notice that for all $k\geq2$,
\[
\log\widetilde a_k
=\log h(\lambda_k)
=\psi(\log\lambda_k)
=\psi(\log\lambda_{k-1}+\log\beta^2+\lambda'_k)
\geq\psi(\lambda'_k)
\]
Let $n=m$ or $n=m+1$ so that $\lambda'_n<0$ and hence
$\lambda'_{n+2k}<0$ for all $k$ and
$\lambda'_{n+2k}\rightarrow-\infty$ for $k\uparrow\infty$. By the
properties of $\psi$, there exists $x_0$ such that $x\leq
x_0\Rightarrow \psi(x)\geq -x-1$. We deduce that, for
$k$ large enough:
\[
\log\widetilde a_{n+2k}
\geq\psi(\lambda'_{n+2k})
\geq-\lambda'_{n+2k}-1
\geq c(1+d)^{2k}|\lambda'_n|-1
\]
so the power series $\sum_k\widetilde a_k z^k$ has zero radius of
convergence and the same for $\sum_k a_k z^k$.
\end{proof}

\paragraph{Proof of Theorem~\ref{thm:ss_2}.}

By Theorem~\ref{thm:ss_3}, we are in the first case iff $\lambda'_1=0$
that is if $\widetilde a_1=h(\beta^2R)$. In the first case $\{a_n\}\in
H$, since $a_n=C_{n_0}\beta^n+O(\beta^{3n})$ and $\beta<1$. In the
second case of course the sequence is not in $H$.

\section{Appendix}

\subsection{Proof of Theorem~\ref{teo existence 1}}

The first claim is an obvious consequence of 
identity~\eqref{variation of constants}.

Let us prove the existence statement. Given $X^{0}\in\mathbb{R}_{+}%
^{\mathbb{N}}$, consider the finite dimensional system
\[
\left\{
\begin{array}
[c]{ll}%
X_{0}(t)=0 & \qquad\forall t\geq0\\
\dot{X_{n}}(t)=k_{n-1}X_{n-1}^{2}(t)-k_{n}X_{n}(t)X_{n+1}(t) & \qquad\forall
t\geq0,\quad\forall n\in\{1,2,\dots,N\}\\
X_{N+1}(t)=0 & \qquad\forall t\geq0\\
X_{n}(0)=X_{n}^{0} & \qquad\forall n\in\{1,2,\dots,N\}
\end{array}
\right.
\]
It has a unique solution: local existence and uniqueness comes from Cauchy
theorem (the vector field on the right-hand-side is locally Lipschitz
continuous), global existence is a consequence of the bound on maximal
solutions derived from
\begin{equation}
\sum_{n=1}^{N}X_{n}^{2}(t)=\sum_{n=1}^{N}X_{n}^{2}(0)
\label{finite dim energy conservation}%
\end{equation}
which is easily proved by computing $\frac{d}{dt}\sum_{n=1}^{N}X_{n}^{2}(t)$.
Denote by $X^{N}$ the unique solution.

Since $X^{0}\in\mathbb{R}_{+}^{\mathbb{N}}$, the solution is positive, in the
sense that $X_{n}^{N}(t)\geq0$ for every $n=1,...,N$ and $t\geq0$. This simply
happens because
\[
X_{n}^{N}(t)=e^{-k_{n}\int_{0}^{t}X^N_{n+1}(r)dr}X_{n}^{0}+\int_{0}^{t}%
e^{-k_{n}\int_{s}^{t}X^N_{n+1}(r)dr}k_{n-1}\left(X^N_{n-1}(s)\right)^{2}ds.
\]

Moreover, for the solution $X^{N}$, for every $n\in\left\{  1,...,N\right\}  $
we have
\[
\frac{d}{dt}\sum_{j=1}^{n}\left(  X_{j}^{N}\right)  ^{2}=2\sum_{j=1}%
^{n}k_{j-1}\left(  X_{j-1}^{N}\right)  ^{2}X_{j}^{N}-k_{j}\left(  X_{j}%
^{N}\right)  ^{2}X_{j+1}^{N}=-k_{n}\left(  X_{n}^{N}\right)  ^{2}X_{n+1}^{N}.
\]
Hence, being positive, we get $\frac{d}{dt}\sum_{j=1}^{n}\left(  X_{j}%
^{N}\right)  ^{2}\leq0$, namely
\[
\sum_{j=1}^{n}\left(  X_{j}^{N}\left(  t\right)  \right)  ^{2}\leq\sum
_{j=1}^{n}\left(  X_{j}^{0}\right)  ^{2}\text{ for all }t\geq0.
\]
In particular,
\begin{equation}
0\leq X_{n}^{N}\left(  t\right)  \leq\sqrt{\sum_{j=1}^{n}\left(  X_{j}%
^{0}\right)  ^{2}}\text{ for all }t\geq0. \label{equiboundedness for AA}%
\end{equation}

On a bounded interval $[0,T]$, consider now the family $\left(  X_{n}%
^{N}\right)  _{N>n}$, for a given $n\geq1$. The assumptions of
Ascoli-Arzel\`{a} theorem are satisfied for this family. Indeed,
equi-boundedness has been proved above in~\eqref{equiboundedness for AA};
equi-uniform-continuity (equi-Lipschitz continuity, in fact) comes from the
identity
\[
X_{n}^{N}(t)-X_{n}^{N}(s)
=\int_{s}^{t}\left[k_{n-1}\bigl(X_{n-1}^{N}(r)\bigr)^2-k_{n}X_{n}^{N}(r)X_{n+1}^{N}(r)\right]dr
\]
and the already proved equi-boundedness (recall always that $n$ is fixed).
Thus, from Ascoli-Arzel\`{a} theorem, for every $n$ there exists a sequence
$\{  N_{k}^{(n)}\}  _{k\in\mathbb{N}}$ such that $\{
X_{n}^{N_{k}^{(n)}}\}  _{k\in\mathbb{N}}$ converges uniformly to a
continuous function $X_{n}$. By a diagonal procedure, one can modify the
previous extraction procedure and get a single sequence $\left(  N_{k}\right)
_{k\in\mathbb{N}}$ such that $\left(  X_{n}^{N_{k}}\right)  _{k\in\mathbb{N}}$
converges uniformly to $X_{n}$. We can thus pass to the limit, as
$k\rightarrow\infty$, in the equation
\[
X_{n}^{N_{k}}(t)
=X_{n}^{0}+\int_{0}^{t}\left[k_{n-1}\bigl(X_{n-1}^{N_{k}}(r)\bigr)^2-k_{n}X_{n}^{N_{k}}(r)X_{n+1}^{N_{k}}(r)\right]dr
\]
and prove that
\[
X_{n}(t)
=X_{n}^{0}+\int_{0}^{t}\left(k_{n-1}X^2_{n-1}(r)-k_{n}X_{n}(r)X_{n+1}(r)\right)dr.
\]
Thus the functions $X_{n}(\cdot)$ are continuously differentiable and satisfy
system~\eqref{model}. Of course they are non negative, being the uniform limit
of non negative functions. Continuation from an arbitrary bounded time
interval to all $t\geq0$ is classical.

Finally, we have to prove properties (i) and (ii) of the theorem, for
\textit{any} positive componentwise solution $X$. As to (i), we repeat
the argument used above for $X^{N}$: for every $n\geq1$, from
system~\eqref{model}, we have
\[
\frac{d}{dt}\sum_{j=1}^{n}X_{j} ^{2}=2\sum_{j=1}^{n}%
k_{j-1}  X_{j-1}  ^{2}X_{j}-k_{j} X_{j}
^{2}X_{j+1}=-k_{n} X_{n} ^{2}X_{n+1}%
\]
hence, $X$ being positive, we get $\frac{d}{dt}\sum_{j=1}^{n}\left(  X_{j}%
^{N}\right)  ^{2}\leq0$ and thus (i) is proved.

As to (ii), from~\eqref{variation of constants} we have
\[
X_{n}(t)\geq X_{n}^{0}e^{-k_{n}\int_{0}^{t}X_{n+1}(r)dr}>0.
\]
For $X_{n+1}(t)$ we use the inequality (again a consequence 
of~\eqref{variation of constants})
\[
X_{n+1}(t)\geq\int_{0}^{t}e^{-k_{n+1}\int_{s}^{t}X_{n+2}(r)dr}k_{n}X_{n}%
^{2}\left(  s\right)  ds
\]
and the fact that $X_{n}(t)>0$ for all $t>0$. By induction we get the result.
The proof of the theorem is complete.

\subsection{Proof of Theorem~\ref{teo existence 2}}

We introduce the same finite dimensional system as above. We do not have
anymore the positivity property, but we still 
have~\eqref{finite dim energy conservation}. Since the (global)
initial energy is finite, we have
\[
\sum_{n=1}^{N}\left(  X_{n}^{N}\left(  t\right)  \right)  ^{2}\leq\left|
X^{0}\right|  _{H}^{2}%
\]
for all $t\geq0$ and $N\geq1$. This implies again a bound on single
components:
\[
\left|  X_{n}^{N}\left(  t\right)  \right|  \leq\left|  X^{0}\right|  _{H}%
\]
for every $t\geq0$, $N\geq1$ and $n=1,...,N$. Having this bound, we proceed as
above and prove, on a given $\left[  0,T\right]  $, the existence of a
componentwise solution $X$, with $X_{n}^{N_{k}}\rightarrow X_{n}$ uniformly on
$\left[  0,T\right]  $ as $k\rightarrow\infty$, along a diverging sequence
\bigskip$\left(  N_{k}\right)  _{k\in\mathbb{N}}$. From
\[
\sum_{n=1}^{N_{k}}\left(  X_{n}^{N_{k}}\left(  t\right)  \right)  ^{2}%
\leq\left|  X^{0}\right|  _{H}^{2}%
\]
one easily get
\[
\sum_{n=1}^{\infty}\left(  X_{n}\left(  t\right)  \right)  ^{2}\leq\left|
X^{0}\right|  _{H}^{2}%
\]
first by truncating the sum up to a given value $R$ and taking the limit in
$k$, then sending $R$ to infinity. Hence in particular $X\left(  t\right)  \in
H $ for all $t\geq0$.

Finally, assume that $X^{0}\in H\cap\mathbb{R}_{+}^{\mathbb{N}}$ and let $X$
be any componentwise solution. From Theorem~\ref{teo existence 1} (i), it
satisfies
\[
\sum_{j=1}^{n}\left(  X_{j}\left(  t\right)  \right)  ^{2}\leq\sum_{j=1}%
^{n}\left(  X_{j}^{0}\right)  ^{2}\leq\left|  X^{0}\right|  _{H}^{2}%
\]
for every $n\geq1$ and $t\geq0$. This implies that $X$ is finite energy and
satisfies~\eqref{energy inequality}. The proof is complete.

\subsection{Estimates on $\alpha_{k,i}$}

Recall the definition of $\alpha_{k,i}$:
\[
\alpha_{k,i}=\begin{cases}
\dfrac{1+\beta^4-\beta^{-1}}{1-\beta^{7}-\beta^{11}} & k=0,\quad i=0\\[2ex]
\dfrac{(1-\beta^{2k+2})(1+\beta^{2k+7})}{1-\beta^{2k+7}-\beta^{4k+11}} & k>0,\quad i=0,k \\[2ex]
\dfrac{1-\beta^{3k+6}\cosh((k-2i)\log\beta)}{1-\beta^{2k+7}-\beta^{4k+11}} & 0<i<k\end{cases}
\]
We study their behaviour for $k$ large and look for bounds. Notice
that the denominators are all positive for $k\geq0$.
\begin{lemma}
There exists $C>0$ such that for all $k\geq0$ and all $0\leq i\leq k$,
\begin{equation}
|1-\alpha_{k,i}|\leq C\beta^{2k}.
\end{equation}
\end{lemma}

\begin{proof} There are three cases.
\begin{enumerate}[1)]
\item When $k\geq 1$ and $i=0$ or $i=k$, $\alpha_{k,i}<1$, in fact:
\begin{align*}
1-\alpha_{k,i}
&=\dfrac{-\beta^{2k+7}-\beta^{4k+11}+\beta^{2k+2}-\beta^{2k+7}+\beta^{4k+9}}{1-\beta^{2k+7}-\beta^{4k+11}}\\
&=\beta^{2k}(\beta^2-\beta^4)\dfrac{1+\beta^{2k+7}}{1-\beta^{2k+7}-\beta^{4k+11}}
\end{align*}
which is positive. The last factor is decreasing in $k$ and is equal
to $4/3$ when $k=1$, so for $k\geq1$
\[
0<\beta^2-\beta^4\leq \beta^{-2k}(1-\alpha_{k,i})\leq4/3(\beta^2-\beta^4)=:C
\]

\item When $k\geq2$, $0<i<k$, we have $\alpha_{k,i}>1$, since,
recalling that $\cosh(x)\leq e^{|x|}$,
\begin{align*}
\alpha_{k,i}-1
&=\dfrac{-\beta^{3k+6}\cosh((k-2i)\log\beta)+\beta^{2k+7}+\beta^{4k+11}}{1-\beta^{2k+7}-\beta^{4k+11}}\\
&\geq\dfrac{-\beta^{3k+6}\beta^{-|k-2i|}+\beta^{2k+7}+\beta^{4k+11}}{1-\beta^{2k+7}-\beta^{4k+11}}\\
&\geq\dfrac{-\beta^{2k+8}+\beta^{2k+7}+\beta^{4k+11}}{1-\beta^{2k+7}-\beta^{4k+11}}\\
&=\beta^{2k+7}\dfrac{1-\beta+\beta^{2k+4}}{1-\beta^{2k+7}-\beta^{4k+11}}
\end{align*}
The right factor is decreasing in $k$, hence for $k\geq2$
\[
\beta^{-2k}(\alpha_{k,i}-1)\geq\beta^7-\beta^8>0
\]

From $\cosh(x)\geq1$ it follows that
\begin{align*}
\alpha_{k,i}-1
&\leq\dfrac{-\beta^{3k+6}+\beta^{2k+7}+\beta^{4k+11}}{1-\beta^{2k+7}-\beta^{4k+11}}\\
&=\beta^{2k+7}\dfrac{1-\beta^{k-1}+\beta^{2k+4}}{1-\beta^{2k+7}-\beta^{4k+11}}
\end{align*}
The right factor is increasing in $k$ for $k\geq2$ and it tends to 1, so
\[
0<\beta^7-\beta^8\leq\beta^{-2k}(\alpha_{k,i}-1)
\leq\beta^7<C
\]
\end{enumerate}
The third case, namely $k=0$ is included by adjusting $C$.
\end{proof}

\subsection{Estimates on the convergence radius of $\hat g$ and $g$}

By equation~\eqref{eq:algeb_g}, in order to find the radius of
convergence of $g$, we need to know the zeros of $\tilde
g(z):=1-2z\hat g(z)$.

\begin{theorem}\label{thm:zero_g_tilde}
The complex function $\tilde g$ has only one root inside the unit
circle centered in the origin. This root is a real number
$R\in(4/5,1)$.
\end{theorem}

The proof is based on Rouch\'e's theorem for holomorphic functions.

We split $\tilde g$ in the sum of two terms: $\tilde g=\tilde
g_A+\tilde g_B$, where $\tilde g_A(z)=1-2d'_0z-2d'_1z^2$ and $\tilde
g_B(z)=-2z\sum_{k=2}^\infty d'_kz^k$. Rouch\'e's theorem tells us that
if $|\tilde g_A|>|\tilde g_B|$ on the contour of some compact set of
$\mathbb C$, then $\tilde g_A$ and $\tilde g$ have the same number of
zeros inside the compact.

We will prove that $|\tilde g_A|>|\tilde g_B|$ on the border of two
circles with centre in the origin and radii 1 and $4/5$.

In the following we first give a function $G$ that is an upper bound
for $g$, then we study separately the maximum of $|\tilde g_B|$ and
the minimum of $|\tilde g_A|$ on the two circles. After that we prove
the theorem.

\paragraph{An upper bound for $g$.}
Let
\[
M:=\frac12\max_{k,i}\alpha_{k,i}=\alpha_{2,1}/2
\]
(It follows from the fact that $\max_i\alpha_{k,i}=\alpha_{k,[k/2]}$
and that $\alpha_{k,[k/2]}$ is decreasing in $k$.)

We introduce another auxiliary sequence $\{D_k\}_k$.
\begin{equation}\label{Dk_recur}
\begin{cases}
D_{0}=d_0\\
D_{1}=d_1\\
D_{k+1}=M\sum_{i=0}^{k}D_{i}D_{k-i}, &\qquad k\geq1,
\end{cases}
\end{equation}

\begin{lemma}\label{thm:lem_raggio_g}
Let $G(z):=\sum_{k=0}^\infty D_kz^k$.  The radius of convergence of
$g$ is larger than the radius of convergence of $G$ which is larger
than $\beta^2$.
\end{lemma}
\begin{proof}
Thanks to~\eqref{eq:recur_d_bis}, by induction $0\leq d_k\leq D_k$ for
all $k$.

By adding the third one of~\eqref{Dk_recur} we find
\[
G(z)-d_0-d_1z
=\sum_{k\geq1}D_{k+1}z^{k+1}
=Mz\sum_{k\geq1}\sum_{i=0}^{k}D_{i}z^iD_{k-i}z^{k-i}
=MzG(z)^2-Md_0^2z
\]
hence for $z\in\mathbb C$ inside the convergence radius of the series:
\[
G(z)=\frac{1\pm\sqrt{1-4Mz(d_0+(d_1-Md_0^2)z)}}{2Mz}
\]
The right sign is `$-$' for all $z$ since $G(0)=d_0$ and $G$ is
continuous. After some algebraic manipulations we write
\begin{equation}\label{G_esplicita}
G(z)=\frac{1-\sqrt{1-4Md_0z+4AM^2d_0^2z^2}}{2Mz}
\end{equation}
where $A=1-\alpha_{0,0}/2M$.

The radius of convergence of $G$ is the distance from the origin of
the nearest zero of the radicand. The roots of the latter are both
real and positive:
\[
z_{1,2}=\frac{1\pm\sqrt{\alpha_{0,0}/2M}}{2AMd_0}
\]
The smallest one is the convergence radius of $G$ and one can verify
that
\begin{equation}\label{eq:z1}
z_1
=\frac{1-\sqrt{\alpha_{0,0}/\alpha_{2,1}}}{(\alpha_{2,1}-\alpha_{0,0})d_0}
>\beta^2
\end{equation}
This completes the proof.
\end{proof}

\paragraph{On the convergence of $\hat g$, $\tilde g$ 
and the maximum of $\tilde g_B$.}

\begin{lemma}\label{thm:lem_g_tilde_b}
The radius of convergence of $\hat g$ and hence $\tilde g$ is greater
than 1. $\tilde g_B$ on the circumferences of radii 1 and $4/5$ is
bounded by:
\begin{align*}
|\tilde g_B(e^{i\theta})|&\leq0.062, \qquad \forall\theta\in[0,2\pi]\\
|\tilde g_B(4/5e^{i\theta})|&\leq0.031, \qquad \forall\theta\in[0,2\pi] 
\end{align*}
\end{lemma}
\begin{proof}
Let $z\in\mathbb C$,
\[
|\tilde g_B(z)|
\leq2|z|\sum_{k=2}^\infty |d'_k||z|^k
=2\sum_{k=2}^\infty |d'_k||z|^{k+1}
\]

By the estimates on $\alpha_{k,i}$, we have
$|1-\alpha_{k,i}|<C\beta^{2k}$; we also need some lower bound and, for
$k\geq1$, $\alpha_{k,i}\geq\alpha_{1,0}$. Then for all $k\geq1$:
\begin{multline*}
|d'_{k+1}|
=\left|\frac12\sum_{i=0}^{k}(\alpha_{k,i}-1)d_{i}d_{k-i}\right|
\leq\frac12C\beta^{2k}\sum_{i=0}^{k}d_{i}d_{k-i}\\
\leq\frac {C\beta^{2k}}{\alpha_{1,0}}\frac12\sum_{i=0}^{k}\alpha_{k,i}d_{i}d_{k-i}
=\frac {C\beta^{2k}}{\alpha_{1,0}}d_{k+1}
\leq\frac {C\beta^{2k}}{\alpha_{1,0}}D_{k+1}
\end{multline*}
Thanks to~\eqref{eq:z1}, this proves that the radius of convergence of
$\hat g$ is greater than 1. 

Putting together the last two upper bounds we get
\begin{multline*}
|\tilde g_B(z)|
\leq2\sum_{k=2}^\infty \frac {C\beta^{2k-2}}{\alpha_{1,0}}D_k|z|^{k+1}
=\frac {C\beta^{-5}}{\alpha_{1,0}}\sum_{k=2}^\infty \beta^{2k}|z|^{k+1}D_k\\
=\frac {C\beta^{-5}|z|}{\alpha_{1,0}}\left(G(\beta^2|z|)-D_0-D_1\beta^2|z|\right)
\end{multline*}
The value of $G(\beta^2|z|)$ can be computed thanks
to~\eqref{G_esplicita}, and the bounds one gets are those in the
lemma.
\end{proof}

\paragraph{On the minimum of $\tilde g_A$.}

\begin{lemma}\label{thm:lem_g_tilde_a}
The maximum and the minimum of $|\tilde g_A(z)|$ on the circumferences
with centre in the origin and radii 1 and $4/5$ are on the following
points:
\begin{align*}
\tilde g_A(-1)&\approx3.170 & 
\tilde g_A(1)&\approx-0.092 \\ 
\tilde g_A(-4/5)&\approx2.650 & 
\tilde g_A(4/5)&\approx0.040. 
\end{align*}
\end{lemma}

\begin{proof}
We study $|\tilde g_A(\rho e^{i\theta})|$. Recall that $\tilde
g_A(z)=1-2d_0'z-2d_1'z^2=\sum_{k=0}^2c_kz^k$ is a polynomial with real
coefficients.
\begin{multline*}
|\tilde g_A(\rho e^{i\theta})|^2
=\tilde g_A(\rho e^{i\theta})\overline{\tilde g_A(\rho e^{i\theta})}
=\tilde g_A(\rho e^{i\theta})\tilde g_A(\rho e^{-i\theta})
=\sum_{j,k=0}^2c_jc_k\rho^{j+k} e^{i\theta (j-k)}\\
=\sum_{j=0}^2c_j^2\rho^{2j}
+2(c_0c_1\rho+c_1c_2\rho^3)\cos(\theta)+2c_0c_2\rho^2 \cos(2\theta)
\end{multline*}
By setting $\frac\partial{\partial\theta}|\tilde g_A(\rho e^{i\theta})|^2=0$ we find:
\begin{align*}
0&=(c_0c_1+c_1c_2\rho^2)\sin(\theta)+2c_0c_2\rho\sin(2\theta)\\
&=(c_0c_1+c_1c_2\rho^2+4c_0c_2\rho\cos(\theta))\sin(\theta)
\end{align*}
The first factor is never 0 because
$|c_0c_1+c_1c_2\rho^2|-|4c_0c_2\rho|>0$
\begin{multline*}
|c_0c_1+c_1c_2\rho^2|-|4c_0c_2\rho|
=2d'_0|1-2d_1'\rho^2|-8|d'_1|\rho\\
=2d_0(1+(1-\alpha_{0,0})d_0^2\rho^2)-4(1-\alpha_{0,0})d_0^2\rho\\
=2d_0[1+(1-\alpha_{0,0})d_0^2\rho^2-2(1-\alpha_{0,0})d_0\rho]\\
>2d_0[1+(1-\alpha_{0,0})^2d_0^2\rho^2-2(1-\alpha_{0,0})d_0\rho]\\
=2d_0(1-(1-\alpha_{0,0})d_0\rho)^2
\geq0
\end{multline*}
So the maximum and the minimum of $|\tilde g_A(z)|$ on the
circumference with centre in the origin and radius $\rho$ are on the
real points $\pm\rho$.

Setting $\rho=1$ and $\rho=4/5$ and computing values we get the
thesis.
\end{proof}

\bigskip

\paragraph{Proof of Theorem~\ref{thm:zero_g_tilde}.}

Direct computation shows that $\tilde g_A$ has only one root inside
the unit circle, in some real point in $(4/5,1)$. By Rouch\'e's
theorem and the bounds of lemmas~\ref{thm:lem_g_tilde_b} 
and~\ref{thm:lem_g_tilde_a} above, $\tilde g$ has only one root inside
the circular crown with centre in the origin and radii $4/5$ and
1. This root is a real number $R\in(4/5,1)$, since the same estimates
state that (by continuity) $\tilde g$ must be zero for some real
number in the interval:
\[
\tilde g(1)
=\tilde g_A(1)+\tilde g_B(1)
<0<\tilde g_A(4/5)+\tilde g_B(4/5)
=\tilde g(4/5).
\]


\begin{thebibliography}{99}
\bibitem{BBBF}D. Barbato, M. Barsanti, H. Bessaih, F. Flandoli, Some rigorous
results on a stochastic GOY model, \textit{J. Stat. Phys.} \textbf{125}
(2006), no. 3, 677--716

\bibitem {BLPTiti}R. Benzi, B. Levant, I. Procaccia, E.S. Titi, Statistical
properties of nonlinear shell models of turbulence from linear advection
model: rigorous results, \textit{Nonlinearity} \textbf{20} (2007), no. 6, 1431--1441, arXiv:nlin/0612033..

\bibitem {biferale}L. Biferale, Shell Models of Energy Cascade in
Turbulence, \textit{Annu. Rev. Fluid. Mech.}, \textbf{35}, (2003), 441-468.

\bibitem {Ces}A. Cheskidov, Blow-up in finite time for the dyadic model of the
Navier-Stokes equations, \textit{Trans. Amer. Math. Soc. }\textbf{360} (2008),
no. 10, 5101--5120, arXiv:math/0601074.

\bibitem {CFP1}A. Cheskidov, S. Friedlander, N. Pavlovic, Inviscid dyadic
model of turbulence: the fixed point and Onsager's conjecture, \textit{J.
Math. Phys.} \textbf{48} (2007), no. 6, 065503, 16 pp, arXiv:math/0610814.

\bibitem {CFP2}A. Cheskidov, S. Friedlander, N. Pavlovic, An inviscid dyadic
model of turbulence: the global attractor, arXiv:math.AP/0610815.

\bibitem {constantin-titi3}P. Constantin, B. Levant, E. S. Titi, Regularity of
inviscid shell models of turbulence, \textit{Phys. Review E} (3), \textbf{75},
no. 1 (2007), 016304, 1-10, arXiv:physics/0607060.

\bibitem {FriPav}S. Friedlander, N. Pavlovic, Blowup in a three-dimensional
vector model for the Euler equations, \textit{Comm. Pure Appl. Math.}
\textbf{57} (2004), no. 6, 705--725.

\bibitem {Fri}U. Frisch, \textit{Turbulence}, Cambridge University Press,
Cambridge (1995).

\bibitem {Ga}G. Gallavotti, \textit{Foundations of Fluid Dynamics}, Texts and
Monographs in Physics, Springer-Verlag, Berlin, 2002.

\bibitem {KatPav}N. H. Katz, N. Pavlovic, Finite time blow-up for a dyadic
model of the Euler equations, \textit{Trans. Amer. Math. Soc.} \textbf{357}
(2005), no. 2, 695--708.

\bibitem {KiZlat}A. Kiselev, A. Zlato\v{s}, On discrete models of the Euler
equation, \textit{IMRN} \textbf{38} (2005), no. 38, 2315-2339, arXiv:math/0507129.

\bibitem {K41}A. N. Kolmogorov, The local structure of turbulence in
incompressible viscous fluids at very large Reynolds numbers, \textit{Dokl.
Akad. Nauk. SSSR} \textbf{30} (1941), 301-305.

\bibitem {Wal}F. Waleffe, On some dyadic models of the Euler equations,
\textit{Proc. Amer. Math. Soc.} \textbf{134} (2006), 2913-2922, arXiv:math/0410380.
\end{thebibliography}
\end{document}